\documentclass[11pt,a4paper]{amsart}
\usepackage[english]{babel}
\usepackage{amsmath,amsthm,amssymb,mathrsfs}

\usepackage{aliascnt}
\usepackage{esint} 

\usepackage{enumitem}
\usepackage{multicol}
\usepackage{graphicx}
\usepackage{float}
\usepackage{xcolor}
\usepackage{tcolorbox}
\usepackage{url}
\usepackage{breakurl}
\usepackage{microtype}
\microtypesetup{expansion=false}

\usepackage[colorlinks=true,linkcolor=blue,citecolor=red]{hyperref}

\usepackage[nameinlink]{cleveref}

\allowdisplaybreaks

\numberwithin{equation}{section}

\theoremstyle{plain}
\newtheorem{theorem}{Theorem}[section]

\newaliascnt{lemma}{theorem}
\newtheorem{lemma}[lemma]{Lemma}
\aliascntresetthe{lemma}
\crefname{lemma}{lemma}{lemmas}

\newaliascnt{proposition}{theorem}
\newtheorem{proposition}[proposition]{Proposition}
\aliascntresetthe{proposition}
\crefname{proposition}{proposition}{propositions}

\newaliascnt{corollary}{theorem}
\newtheorem{corollary}[corollary]{Corollary}
\aliascntresetthe{corollary}
\crefname{corollary}{corollary}{corollaries}

\theoremstyle{definition}

\newaliascnt{definition}{theorem}
\newtheorem{definition}[definition]{Definition}
\aliascntresetthe{definition}
\crefname{definition}{definition}{definitions}

\newaliascnt{example}{theorem}
\newtheorem{example}[example]{Example}
\aliascntresetthe{example}
\crefname{example}{example}{examples}

\newaliascnt{assumption}{theorem}

\aliascntresetthe{assumption}
\crefname{assumption}{assumption}{assumptions}

\newaliascnt{question}{theorem}

\aliascntresetthe{question}
\crefname{question}{question}{questions}

\theoremstyle{remark}

\newaliascnt{remark}{theorem}
\newtheorem{remark}[remark]{Remark}
\aliascntresetthe{remark}
\crefname{remark}{remark}{remarks}

\newaliascnt{observation}{theorem}

\aliascntresetthe{observation}
\crefname{observation}{observation}{observations}

\newaliascnt{convention}{theorem}

\aliascntresetthe{convention}
\crefname{convention}{convention}{conventions}

\crefformat{equation}{(#2#1#3)}
\Crefformat{equation}{(#2#1#3)}

\crefrangeformat{equation}{(#3#1#4)--(#5#2#6)}
\Crefrangeformat{equation}{(#3#1#4)--(#5#2#6)}

\crefmultiformat{equation}{(#2#1#3)}{ and~(#2#1#3)}{, (#2#1#3)}{ and~(#2#1#3)}
\Crefmultiformat{equation}{(#2#1#3)}{ and~(#2#1#3)}{, (#2#1#3)}{ and~(#2#1#3)}

\newcommand{\RR}{\mathbb{R}}

\newcommand{\NN}{\mathbb{N}}

\newcommand{\BB}{\mathbb{B}}
\newcommand{\Om}{\Omega}
\newcommand{\Ombar}{\overline{\Omega}}

\newcommand{\HH}{\mathcal{H}}

\newcommand{\W}[1]{W^{#1}}
\newcommand{\Wdot}[1]{\dot{W}^{#1}}
\newcommand{\M}[1]{M^{#1}}
\newcommand{\Mdot}[1]{\dot{M}^{#1}}

\newcommand{\I}[1]{\int_{#1}}

\newcommand{\Ia}[1]{\fint_{#1}}

\renewcommand{\d}{\,d}

\newcommand{\leqa}{\lesssim}

\newcommand{\diam}{\mathrm{diam}}
\newcommand{\dist}{\mathrm{dist}}
\newcommand{\supp}{\mathrm{supp}}
\newcommand{\loc}{\mathrm{loc}}

\newcommand{\lp}{\left(}
\newcommand{\rp}{\right)}
\newcommand{\lb}{\left[}
\newcommand{\rb}{\right]}
\newcommand{\lc}{\left\{}
\newcommand{\rc}{\right\}}
\newcommand{\abs}[1]{\left|#1\right|}

\begin{document} 
	\title{How to recognise extension domains} 
	
	\author{$^1$Riddhi Mishra}
	\email{$^1$riddhi.r.mishra@jyu.fi}
	\address{$^1$Department of Mathematics and Statistics, University of Jyv\"askyl\"a,  P.O. Box 35, FI-40014, Jyv\"askyl\"a, Finland}
	
	\author{$^2$Kaushik Mohanta} 
	\email{$^2$kaushik@iitgoa.ac.in}
	\address{$^2$School of Mathematics and Computer Science, Indian Institute of Technology Goa, Ponda-403401, Goa, India}
	
	\thanks{The First-named author was supported by the Academy of Finland (Project No. 323960) and the Research Council of Finland Centre of Excellence in Randomness and Structures (Project number 364210). The second-named author was supported by the grant GA\v{C}R P201/24-10505S.}

	
	
	\begin{abstract}
		Let $\Omega \subset \mathbb{R}^n$ be a bounded domain and $1 < p < \infty$. 
		We prove that there is a bounded extension operator $\dot{W}^{1,p}(\Omega)\to \dot{W}^{1,p}(\mathbb{R}^n)$ if and only if $\Omega$ satisfies the measure density condition and a Bourgain-Brezis-Mironescu type inequality (or limiting formula). As a key ingredient, we establish a fractional Poincaré-type inequality 
		under the assumption of Ahlfors regularity alone, improving a result of 
		Ponce (2004). 
		
		We also prove that, under a mild Hausdorff 
		measure condition on the boundary $\partial \Omega$, fractional 
		extension (from $\dot{W}^{1,p}(\Omega)$ to $\dot{W}^{s,p}(\mathbb{R}^n)$) 
		at a single exponent $s > 1/p$ self-improves to full first-order Sobolev 
		extension (from $\dot{W}^{1,p}(\Omega)$ to $\dot{W}^{1,p}(\mathbb{R}^n)$).
		
		These results clarify the role of nonlocal estimates in the geometry of 
		Sobolev extension domains.
	\end{abstract}
	\maketitle 
	
	\section{Introduction} 
	In this article, we aim to study the relation between Sobolev extension properties of a given bounded open set $\Om\subseteq  \RR^n$ and the relation between Sobolev and fractional Sobolev spaces. The problem of finding a function $Ef:\RR^n\to \RR$, for a given function $f:\Om\subseteq  \RR^n\to \RR$ such that $Ef|_{\Om}=f$ and $Ef$ preserves the `regularity' of $f$ is a fundamental theme in analysis.
	
	Formally, given $1\leq p\leq \infty$ and $s\in(0,1]$, we say that $\Om$ is a $(1,s,p)$-extension domain if there exists a bounded operator $E:\Wdot{1,p}(\Om)\to \Wdot{s,p}(\RR^n)$ such that $Ef|_{\Om}=f$ for all $f\in \Wdot{1,p}(\Om)$ -- the homogeneous Sobolev space; for $s\in (0,1)$, the homogeneous fractional Sobolev space $\Wdot{s,p}(\Om)$ can be viewed as the fractional counterpart of $W^{1,p}(\Om)$. One of the main goals of this paper is to characterise bounded $(1,1,p)$-extension domains (written as $(1,p)$-extension domains for brevity) in terms of the relation between $\Wdot{1,p}(\Om)$ and $\Wdot{s,p}(\Om)$ (or $\Wdot{s,p}(\RR^n)$) --the fractional order homogeneous Sobolev spaces. 
	
	The search for characterisations of $(1,p)$-extension domains has spanned several decades, evolving from simple boundary smoothness conditions to deep connections with geometric measure theory and functional inequalities. Calder\'on \cite{calderon} (for $1<p<\infty$) and Stein \cite{stein} (for $p=1$ and $p=\infty$) have shown that Lipschitz domains are $(1,p)$-extension domains for $1\leq p\leq \infty$. Later, Jones \cite{Jones} proved that the more general $(\varepsilon, \delta)$-domains are also $(1,p)$-extension domains. Haj{\l}asz-Koskela-Tuominen \cite{HKT-8} showed that a domain $\Om$ is a $(1,p)$-extension domain if and only if it satisfies measure density condition and $\W{1,p}(\Om)=\M{1,p}(\Om)$ (see \Cref{preli} for definitions). Some other related characterisations were given by the same authors in \cite{HKT}. Some geometric characterisations of $(1,p)$-extension domains can be found in \cite{Jones,Koskela,Shvartsman,KoRaZh}. For more discussion regarding examples of extension domains, see \cite{Mazya2}.
	
	In their seminal work, Bourgain, Brezis, and Mironescu \cite{BoBrMi} established a fundamental connection between nonlocal functionals and classical local derivatives (see also subsequent works such as \cite{Davila,Brezis2002,BBM02}). Specifically, they showed that for a smooth bounded domain $\Om\subseteq  \RR^n$, and $f\in L^p(\Om)$, the following holds:
	$$
	\lim_{s\to 1-} (1-s)\I{\Om}\I{\Om}\frac{|f(x)-f(y)|^p}{|x-y|^{n+sp}}\d y \d x
	= K(n,p) \I{\Om} |\nabla f|^p \d x.
	$$
	Its primary significance lies in its ability to recover the $W^{1,p}$-seminorm as a limit of fractional Gagliardo-type seminorms, effectively bridging the gap between nonlocal and local theories. This identity, commonly referred to as the \textbf{BBM formula} in literature, has become a cornerstone in the study of Sobolev spaces and nonlocal operators \cite{Ponce04,AmDeMa,Nguyen11,Ludwig,DiRoVa,BrJa}.

	A key ingredient in the proof of the BBM formula is the accompanying \textbf{BBM inequality}: For a fixed $s_0\in(0,1)$, $s\in (s_0,1)$, and $f\in \W{1,p}(\Om)$, we have
	$$
	(1-s)\I{\Om}\I{\Om}\frac{|f(x)-f(y)|^p}{|x-y|^{n+sp}}\d y \d x
	\leq C(n,p,\Om) \I{\Om} |\nabla f|^p \d x.
	$$
	While the original formulation in \cite{BoBrMi} includes an additional $\|f\|_{L^p(\Om)}$ term on the right-hand side, it can be ignored in view of \cite[Corollary 4.9]{Heronkoskela} which says that a bounded domain is a nonhomogeneous Sobolev extension domain if and only if it is a homogeneous Sobolev extension domain
	
	While the original result of Bourgain-Brezis-Mironescu  was stated for smooth domains, the role of the domain's geometry remained a subject of significant interest. For instance, it was later observed by Drelichman and Duran \cite{DreDur} that the BBM-type limit can be recovered for an arbitrary bounded domain if the nonlocal seminorm is suitably restricted to a neighborhood of the diagonal ($x=y$), thereby bypassing the need for the extension property. These developments lead to a fundamental question: to what extent does the validity of the BBM inequality characterise the geometry of the domain? Specifically, one might ask whether the BBM inequality (or BBM formula) for $\Om$ necessarily implies that $\Om$ is a $(1,p)$-extension domain. In this article, we show that this is false for general domains (see \Cref{rm,cusp-example} below), but if we restrict ourselves within the class of domains satisfying the measure density condition, BBM inequality, BBM formula and $(1,p)$-extension property-- all becomes equivalent.

	We now proceed to state our first main result	
	\begin{theorem}\label{th-main} 
		Let $\Om \subseteq \RR^n$ be a bounded open set, $p\in (1,\infty)$. Then the following are equivalent\\ 
		\textbf{\emph{(1)}} $\Om$ is a $(1,p)$-extension domain.\\ 
		\textbf{\emph{(2)}} $\Om$ satisfies measure density condition, and there exists $s_0\in(0,1)$ such that
		\begin{equation}\label{bbm-hom}
			(1-s)[f]_{W^{s,p}(\Om)}^p
			\leqa [f]_{W^{1,p}(\Om)}^p,
		\end{equation}
		for any $s\in(s_0,1)$ and $f\in \Wdot{1,p}(\Om)$.\\
		\textbf{\emph{(3)}} $\Om$ satisfies measure density condition, and there is some $s_0\in(0,1)$ such that
		\begin{equation}\label{bbm}
			(1-s)[f]_{W^{s,p}(\Om)}^p
			\leqa \|f\|_{W^{1,p}(\Om)}^p,
		\end{equation}
		for any $s\in (s_0,1)$ and $f\in \W{1,p}(\Om)$.\\
		\textbf{\emph{(4)}} $\Om$ satisfies measure density condition, and
		\begin{equation}\label{bbm-lim-h}
			\lim\limits_{s\to 1-}(1-s)[f]_{W^{s,p}(\Om)}^p
			=K(n,p) [f]_{W^{1,p}(\Om)}^p,
		\end{equation}
		for any $f\in \Wdot{1,p}(\Om)$.\\
		\textbf{\emph{(5)}} $\Om$ satisfies measure density condition, and
		\begin{equation}\label{bbm-lim-n}
			\lim\limits_{s\to 1-}(1-s)[f]_{\W{s,p}(\Om)}^p
			=K(n,p) [f]_{\W{1,p}(\Om)}^p,
		\end{equation}
		for any $f\in \W{1,p}(\Om)$.\\
	\end{theorem}
	For the definition of measure density condition, see \Cref{measure-density-def}.
	\begin{remark}\label{rm}
		In \Cref{th-main}, in the implication \textbf{(2)} $\implies$ \textbf{(1)}, neither of hypotheses measure density condition \cref{measure-density} and BBM inequality \cref{bbm} can be dropped. To understand this, we look at two examples: the interior cusp and the exterior cusp. We know that both are non-extension domains. In \cite{NePaEn}, it is shown that the interior cusp does not support \cref{bbm} (for $n=2$ and $p>2$), but it can be checked that it satisfies \cref{measure-density}. On the other hand, we have shown in \Cref{cusp-example} that the exterior cusp does not satisfy measure density condition \cref{measure-density}, but it supports BBM inequality \cref{bbm}.
	\end{remark}
	
	\begin{remark}
		The original proof of Bourgain-Brezis-Mironescu follows a path that says: in a $(1,p)$-extension domain, (3) implies (5). However, it is not clear whether (3) is equivalent with (5) or, similarly, (2) with (4). Our \Cref{th-main} clarifies the scenario in the case of domains satisfying the measure density condition.
	\end{remark}
	
	A key ingredient in the proof of \Cref{th-main} is our next result, \Cref{Ponce1}, which can be compared to a result due to Ponce \cite[Theorem 1.1]{ponce}. Both \Cref{Ponce1} and \cite[Theorem 1.1]{ponce} have a similar conclusion, but in \cite{ponce} the hypothesis requires the domain to be a $(1,p)$-extension domain, whereas we require it to only satisfy measure density condition. Recall that there are domains, satisfying measure density condition, which are not $(1,p)$-extension domains, for example, consider a slit-disk $B(0,1)\setminus (\{0\}\times[0,1)$ in $\RR^2$. Thus, \Cref{Ponce1} can be seen as an improvement of \cite[Theorem 1.1]{ponce}, with a proof without using heavy machinery. Also, we prove the result in $W^{s,p}_q$-space setup, which is a generalisation (along the Triebel-Lizorkin scales ) of fractional Sobolev spaces $W^{s,p}$; more information regarding such spaces can be found in \cite{PrSa}.
	\begin{theorem}\label{Ponce1}
		Let $\Om$ be a bounded domain satisfying the measure density condition, $1\leq p \leq q<\infty$, and $s_0\in \lp \max{\{\frac{q-n}{q},0\}},1\rp $. Then, for any $s\in (s_0,1)$ and any $f \in L^{p}(\Om)$, we have
		\begin{equation}
			\Ia{B\cap \Om}|f(x)-f_{B\cap\Om}|^p\d x  \leq C(n,p,q)(1-s)^\frac{p}{q} \diam(B)^{sp}\Ia{B\cap \Om}\lp  \I{B\cap\Om}\frac{|f(x)-f(y)|^q}{|x-y|^{n+sq}}\d x \rp ^\frac{p}{q} \d y 
		\end{equation}
		for any ball $B$.
	\end{theorem}
	As a direct consequence of \Cref{Ponce1}, we can deduce the following, which we list here as it may be of interest to the readers.
	\begin{corollary}
		Let $p\in (1,\infty)$, $s_0\in \lp \max\lc \frac{p-n}{p},0\rc ,1\rp $, and $s\in(s_0,1)$. For a bounded domain $\Om$, for any $f\in L^p(\Om)$,
		\begin{align*}
			\Ia{y\in\Om}  &\abs{ f(y)-f_{\Om}} ^p \d y
			\leq C(n,p)\frac{(1-s)\text{diam}(\Om)^{sp+n}}{|\Om|}  \Ia{y\in\Om} \I{\Om} \frac{|f(y)-f(z)|^p}{|y-z|^{n+sp}}  \d z \d y.
		\end{align*}
	\end{corollary}
	Some interesting results related to \Cref{Ponce1} can be found in \cite{HuMaPeVa,MyPeWe}.

	In \Cref{th-main}, the characterisation requires a family of inequalities as $s \to 1-$. While this is mathematically elegant, from a `practical' perspective, very difficult to verify. So, we ask: what happens if the fractional parameter $s$ is fixed? For this, we use a stronger hypothesis than BBM inequality ((2) or (3) of \Cref{th-main}), but for a fixed $s$, namely the $(1,s,p)$-extension property.

	\begin{theorem}\label{th-self-improve}
		Let $p\in(1,\infty)$ and $s\in\lp \frac{1}{p},1\rp $. Then every bounded $(1,s,p)$-extension domain $\Om$, with $\HH^{n-1+\frac{sp-1}{p}}(\partial \Om)=0$, is also a $(1,p)$-extension domain.
	\end{theorem}
	The lower bound of $s$ is natural, as $W^{s,p}(\RR^n)$, for $sp\leq 1$, is known to contain characteristic functions of certain domains. However, it is not clear to us, to what extent the assumption of $\HH^{n-1+\frac{sp-1}{p}}$ measure of the boundary can be relaxed in the hypothesis of \Cref{th-self-improve}.

	We now discuss the paper's organisation. \Cref{preli} contains necessary definitions, notations, conventions and some known results required for the subsequent sections. In \Cref{EVE}, we prove \Cref{th-main} and \Cref{Ponce1}. Examples of domains satisfying BBM inequality but not $(1,p)$-extension property are given in \Cref{sec-examples}. Finally, in \Cref{se}, we prove \Cref{th-self-improve}.

	\vspace{1cm}
	
	\section{Preliminaries}\label{preli}
	Let $\Om \subseteq \RR^n $ be open and $\Om \neq \RR^n$. By a cube, we shall understand a set of the form $c+\lambda [-1,1]^n$. For a cube $Q$, its diameter is denoted by $\diam{Q}$. We denote by $C$ a generic positive constant that may vary from appearance to appearance. Whenever necessary, we shall emphasise the dependencies of $C$ by writing $C(\cdots)$.
	\begin{definition}\label{measure-density-def}
		A set $\Om\subseteq  \RR^n$ is said to satisfy the measure density condition if
		\begin{equation}\label{measure-density}
			|B(x,r)\cap \Om| \geq C(\Om,n) r^n, \quad \text{for all } x\in \Om \ \text{and } r\in(0,\text{diam}(\Om)/2),
		\end{equation}
	\end{definition}
	\begin{definition}
		Let $\Om$ be an open set in $\RR^n$, $1\leq p<\infty$ and $0<s\leq 1$. Given $f \in L^1_{\loc}(\Om)$, the (fractional, when $s<1$) Sobolev seminorm is given by
		\[[f]_{W^{s, p}(\Om)}:=
		\begin{cases}
			\lp \I\Om \abs{\nabla f(x)}^p\d x \rp ^{\frac{1}{p}}\qquad &\text{ when } s=1\\
			\lp \I{\Om}\I{\Om}\frac{|f(x)-f(y)|^p}{|x-y|^{n+sp}} \d y \d x \ \rp ^\frac{1}{p} \qquad &\text{ when } s<1,
		\end{cases}
		\] 
		where $(\int_{\Om}|\nabla f(x)|^pdx)^{\frac{1}{p}}$ is interpreted to be infinite unless the distributional gradient $\nabla f$ is a $p$-integrable (vector valued) function.
		
		We also define the norm, for $0<s\leq 1$,
		\[
		\|f\|_{W^{s, p}(\Om)}:= \lp \|f\|_{L^p(\Om)}^p + [f]_{W^{s, p}(\Om)}^p\rp ^\frac{1}{p}.
		\] 
		We denote, by $\Wdot{s,p}(\Om)$, the collection of all $f\in L^1_{\loc}(\Om)$ such that $[f]_{W^{s,p}(\Om)}$ is finite. The space $\Wdot{s,p}(\Om)$, equipped with $[\cdot]_{W^{s,p}}$, is then called a homogeneous (fractional, when $s<1$) Sobolev space. The (fractional, when $s<1$) Sobolev space $W^{s,p}(\Om)$ is defined to be $\Wdot{s,p}(\Om)\cap L^p(\Om)$, and it is equipped with the norm $\|\cdot\|_{W^{s,p}}(\Om)$.
		
	\end{definition}
	
	\begin{definition}
		Let $s\in(0,1]$ and let $1< p<\infty$ be given. A bounded open set $\Om\subset\RR^n$ is said to be a $(1,s,p)$-extension domain if there exists a bounded extension operator 
		\begin{equation*}
			E:\Wdot{1, p}(\Om)\to \Wdot{s, p}(\RR^n),  
		\end{equation*}   
		that is, $Ef_{|\Om}=f$ and there exist $C>0$ such that 
		\begin{equation*}
			[Ef]_{W^{s,p}(\RR^n)}\leq C [f]_{W^{1,p}(\Om)}.
		\end{equation*}
		We call a $(1,1,p)$-extension domain simply a $(1,p)$-extension domain.
	\end{definition} 
	\begin{definition}
		Let $f\in L^{1}_{\loc}(\RR^n)$. The maximal function $\mathcal{M}f: \RR^n \to [0,\infty]$ is defined by 
		\begin{equation*}
			\mathcal{M}f(x) := \sup_{Q\ni x}\frac{1}{|Q|}\I{Q}|f(y)|\d y,
		\end{equation*}
		where $Q$ denotes a cube.
	\end{definition}
	\begin{remark}
		In the above definition, one could replace the cubes $Q$ by balls $B(y,r)$ and still get the following result.
	\end{remark}
	\begin{theorem}[see \cite{stein}] \label{HLMI}
		Let $1 < p \leq \infty$. Then there exists a constant $C=C(n,p) > 0$, such that
		\begin{equation*}
			\|\mathcal{M}f\|_{L^p(\RR^n)} 
			\leq C \|f\|_{L^p(\RR^n)},
		\end{equation*}
		for all $f \in L^p(\RR^n)$.
	\end{theorem}
	\begin{definition}
		Let
		$\lambda \in (0,\infty)$ and $\delta \in (0,\infty]$. Given a set $A \subset \RR^n$, a $\delta$-cover of $A$ is any countable family of sets $\{B_j\}_{j\in \NN}$ so that
		$A \subset \cup_{j\geq 1}B_j$ and $\diam(B_j) \leq  \delta$
		for all $j \in \NN.$ The $\lambda$-dimensional Hausdorff $\delta$-content is
		\begin{equation*}
			\HH^{\lambda}_{\delta}(A):= \inf\{\sum_{j\geq 1}\diam(B_j)^\lambda :A\subset \cup_{j\geq 1}B_j, \diam(B_j)\leq \delta\}.
		\end{equation*}
		
		The $\lambda$-dimensional Hausdorff measure is defined as
		\begin{equation*}
			\HH^\lambda (A) := \lim_{\delta\to 0} \HH^{\lambda}_{\delta}(A).
		\end{equation*}
	\end{definition}
	\begin{definition}
		A function $f: [a,b]\to \RR $ is said to be an absolutely continuous function if for each $\varepsilon >0$ there exists a $\delta>0$ such that
		\begin{equation*}
			\sum_{i=1}^{k}|f(b_i)-f(a_i)|<\varepsilon,
		\end{equation*}
		whenever $(a_i,b_i)$ are non-overlapping subintervals of $[a,b]$ with
		\begin{equation*}
			\sum_{i=1}^{k}|b_i-a_i|<\delta.
		\end{equation*}
	\end{definition}
	\begin{definition}
		Let $p>0$ and let $\Om \subset \RR^n$ be a domain. A measurable function $u$ belongs to the homogeneous Haj\l asz Sobolev space($\Mdot{1,p}(\Om)$) if there exists a function $g\in L^{p}(\Om)$ and a set $E\subset \Om$ of measure zero such that
		for all $x,y\in \Om\setminus E$, we have
		\begin{equation}\label{2}
			|u(x)-u(y)|\leq |x-y| (g(x) + g(y)).
		\end{equation}
		The seminorm in $\Mdot{1,p}(\Om)$ is given by
		\begin{equation*}
			\|u\|_{\Mdot{1,p}(\Om)}:= \inf{\|g\|_{L^{p}(\Om)}},
		\end{equation*}
		where the infimum is taken over all nonnegative functions $g$ satisfying \cref{2}.
	\end{definition}
	\begin{definition}[Poincar\'e inequality]
		Let $1\leq p<\infty$, $\Om\subset\RR^n$ be a bounded open set. The inequality 
		\begin{equation*}
			\I\Om|f(x)-f_\Om|^p\d x \leq C(n,p,\Om)\I\Om|\nabla f(x)|^p\d x 
		\end{equation*}
		is referred to as a global $(p, p)$-Poincar\'e inequality. Here we used the notation $f_\Om= \frac{1}{|\Om|}\I\Om f(y)\d y$.
	\end{definition}

	\begin{definition}[Poincar\'e domain]
		A bounded domain $\Om\subset\RR^n$ is called a $p$-Poincar\'e domain for $1\leq p <\infty$, if satisfies the global $(p,p)$-Poincar\'e inequality for every $f \in W^{1, p}(\Om)$.
	\end{definition}
	
	The following Lemmas follow from results in \cite{stein}.
	\begin{lemma}\label{lem:whitney}
		Let $D\subseteq \RR^n$ be an open subset. Then there exists a countable collection $\mathcal{W}(D)$ of closed cubes such that
		\begin{enumerate}
			\item $\bigcup_{Q\in \mathcal{W}(D)}Q= D$,
			\item Any two different cubes in $\mathcal{W}(D)$ have disjoint interiors,
			\item Let $Q\in \mathcal{W}(D)$ and $x\in 5Q$. Then $2\diam{Q}< \dist(x,\RR^n\setminus D)< 8\diam{Q}$,
			\item For each $Q\in \mathcal{W}(D)$, there is $x_Q^{\star}\in \RR^n\setminus D$
			such that $\dist(x_Q^{\star}, Q)<15\diam{Q}$,
			\item There is $M \in \NN$ such that $\sum_{Q\in \mathcal{W}(D)}\chi_{15 Q}(x)\leq M$ for all $x\in D$
			\item Denote, for $x\in D$, $\mathcal{W}_x:=\{Q\in \mathcal{W}(D)\ |\ x\in Q\}$. There is some $C>0$ such that $Q\subseteq C R$ whenever $Q,R\in \mathcal{W}_x$ for some $x\in D$.   
		\end{enumerate}
	\end{lemma}
	We refer to the above collection $\mathcal{W}(D)$ as a Whitney covering of $D$. With such a Whitney covering of $D$, we can associate a Lipschitz partition of unity.
	\begin{lemma}\label{POU}
		Given a Whitney covering  $\mathcal{W}(D)$ of $D$, there exists a partition of unity $\{\varphi_Q\}_{Q \in \mathcal{W}(D)}$ such that
		\begin{enumerate}
			\item $\supp\varphi_Q \subseteq 2Q$,
			\item $\varphi_Q(x) \leq M^{-1}$ for all $x\in Q$,
			\item there is a constant $K>0$ such that each $\varphi_Q$ is $\frac{K}{\diam(Q)}$-Lipschitz,
			\item $\sum_{Q\in \mathcal{W}(D)}\varphi_Q(x)= \chi_{D}(x)$.
		\end{enumerate}
	\end{lemma}

	To state the result, we set some notations. For $w \in \mathbb{S}^{n-1}$, we denote 
	$$
	T_w\mathbb{S}^{n-1}:= \{x\in \RR^n\ |\ x\cdot w=0 \},
	$$ -the collection of normals to $\mathbb{S}^{n-1}$ at $w$. For fixed $w\in \mathbb{S}^{n-1}$ and $x\in T_w\mathbb{S}^{n-1}$, consider the line
	$$
	L_{w,x}:=\lc  (x+sw)\in \Om\ |\ s\in \RR  \rc .
	$$
	
	\begin{lemma}[{\cite[Lemma~2.4]{LoSl}}]
		Let $s\in (0,1)$, $p\in(0,\infty)$, $\Om$ be a open set in $\RR^n$, and $f\in W^{s,p}(\Om)$. Then
		\begin{multline}
			\int_{\Om}\int_{\Om} \frac{|f(x)-f(y)|^p}{|x-y|^{n+sp}} \d y  \d x \\
			= \frac{1}{2} \I{w\in \mathbb{S}^{n-1}} \I{x\in T_w\mathbb{S}^{n-1}} \I{z\in L_{w,x}} \I{y\in L_{w,x}}
			\frac{|f(y)-f(z)|^p}{|y-z|^{1+sp}} \d\HH^{1}(y) \d\HH^{1}(z)  \d\HH^{n-1}(x) \d\HH^{n-1}(w).
		\end{multline}
		That is
		\begin{equation*}
			[f]_{W^{s,p}(\Om)}^p
			= \frac{1}{2} \I{w\in \mathbb{S}^{n-1}} \I{x\in T_w\mathbb{S}^{n-1}} [f]_{W^{s,p}(L_{w,x})}^p \d\HH^{n-1}(x) \d\HH^{n-1}(w).
		\end{equation*}
	\end{lemma}
	\begin{corollary}\label{loss-sloan}
		Let $s\in (0,1)$, $p\in(0,\infty)$, $\Om$ be a open set in $\RR^n$, and $f\in W^{s,p}(\Om)$. Then, for $\HH^{n-1}$-almost every hyperplane $H$, for $\HH^{n-1}$-almost every line $L$ perpendicular to $H$, $f\in W^{s,p}(L\cap\Om)$.
	\end{corollary}
	
	An corollary of \Cref{loss-sloan} follows from the fact that, given any $N\subseteq \mathbb{S}^{n-1}$, we can find an orthonormal set $\{x_1,\cdots,x_n\} \subseteq \mathbb{S}^{n-1}$ such that $\{x_1,\cdots,x_n\} \cap N=\emptyset$. So, we have
	
	\begin{corollary}\label{loss-sloan-cor}
		Let $s\in (0,1)$, $p\in(0,\infty)$, $\Om$ be a open set in $\RR^n$, and $f\in W^{s,p}(\Om)$. Then, there exists a Cartesian coordinate system such that, for $\HH^{n-1}$-almost every line $L$ parallel to the axes, $f\in W^{s,p}(L\cap\Om)$.
	\end{corollary}
	
	Next, we state the ACL characterisation of Sobolev spaces.
	
	\begin{lemma}[ACL characterisation]\label{ACL}
		Let $\Om\subset \RR^n$ be open and let $f\in L^1_{\loc}(\Om)$.
		Then the following are equivalent:
		\begin{enumerate}
			\item $f\in W^{1,1}_{\loc}(\Om)$;
			\item There exists $g\in L^1_{\loc}(\Om)$ such that for each $i=1,\dots,n$,
			for almost every line $\ell$ parallel to the $x_i$-axis, the restriction
			$f|_{\ell\cap\Om}$ is absolutely continuous and
			\[
			\frac{d}{dt} f(x+te_i)= g_i(x+te_i)
			\quad \text{for a.e. } t,
			\]
			where $g=(g_1,\dots,g_n)$.
		\end{enumerate}
		In this case, $g_i=\partial_i f$ in the weak sense.
	\end{lemma}
	
	We also recall the well known result from geometric measure theory.
	\begin{lemma}[\protect{Frostman Lemma, see \cite{Mattila}}]\label{frostman}
		Let $A$ be a Borel subset of $\RR^n$, and let $t > 0$. Then the following are equivalent:
		
		\begin{enumerate}
			\item $\HH^t(A) > 0$.
			\item There is an unsigned Borel measure $\mu$ on $\RR^n$ satisfying $\mu(A) > 0$, such that
			$$
			\mu (B(x,r)) \leq r^t
			$$
			holds for all $x \in \RR^n$ and $r>0$.
		\end{enumerate}
	\end{lemma}
	
	The following two are slightly modified versions of results in \cite {Brezis2002} (also see \cite{BoBrMi}).
	\begin{lemma}\label{lm-pointwise} 
		Let $R>0$, $\BB=B(0,R)$, $p\in (1,\infty)$. If $f\in C^2(\BB)$, then 
		there is a constant $K=K(n,p)$ such that, for any $x\in \BB$, 
		$$ \lim_{s\to1-} (1-s) \I{\BB}\frac{|f(x)-f(y)|^p}{|x-y|^{n+sp}}\d y  =K (n,p) |\nabla f|^p. 
		$$ 
	\end{lemma}

	\begin{lemma}\label{bbm-formula} 
		Let $R>0$, $\BB=B(0,R)$, $p\in (1,\infty)$. If $f\in L^1_{\loc}(\BB)$ is 
		such that 
		$$ 
		L_{n,p,R}(f):=\liminf_{s\to1-} (1-s) \I{\BB}\I{\BB}\frac{|f(x)-f(y)|^p}{| 
			x-y|^{n+sp}}\d y \d x <\infty, 
		$$ 
		then $f$ has a weak derivative and $\nabla f$ such that 
		$$
		\|\nabla f\|_{L^p(\Om)} \leq L_{n,p,R}(f)<\infty .
		$$
	\end{lemma} 
	\vspace{1cm}
	
	\section{Extension Via Embedding}\label{EVE}
	In this section, we sometimes use the following convention for a simpler presentation: for two real numbers, we write $A\leqa B$ when there is some constant $C=C(n,p,\Om)$ such that $A\leq CB$.

	We start this section by proving \Cref{Ponce1}.
	
	\begin{proof}[\textbf{Proof of the \Cref{Ponce1}}]
		Observe that
		\begin{equation*}
			\I{t=0}^r t^{q-sq-1} \d t = \frac{r^{q-sq}}{q(1-s)} .
		\end{equation*}
		Hence, for every $x\in \Om$ and $r>0$,
		\begin{align*}
			&\Ia{y\in B(x,r)\cap\Om} \lp \Ia{z\in B(x,r)\cap\Om}\frac{|f(y)-f(z)|^q}{|y-z|^q} \d z\rp ^{\frac{p}{q}} \d y\\
			&\quad= \Ia{y\in B(x,r)\cap\Om} \lp \frac{q(1-s)}{r^{q-sq}}\I{t=0}^r t^{q-sq-1}  \Ia{z\in B(x,r)\cap\Om} \frac{|f(y)-f(z)|^q}{|y-z|^q} \d z \d t \rp ^{\frac{p}{q}} \d y \\
			&\quad= \frac{(q(1-s))^{\frac{p}{q}}}{|B(x,r)\cap \Om|^{\frac{p}{q}}}\Ia{y\in B(x,r)\cap\Om} \lp \frac{1}{r^{q-sq}}\I{t=0}^r t^{q-sq-1}  \I{z\in B(x,r)\cap\Om} \frac{|f(y)-f(z)|^q}{|y-z|^q} \d z \d t \rp ^{\frac{p}{q}} \d y \\
			&\quad\leq \frac{(q(1-s))^{\frac{p}{q}}}{|B(x,r)\cap \Om|^{\frac{p}{q}}}\Ia{y\in B(x,r)\cap\Om} \lp \frac{1}{r^{q-sq}}\I{t=0}^r t^{q-sq-1}  \I{z\in B(x,r)\cap\Om \cap B(y,t)} \frac{|f(y)-f(z)|^q}{|y-z|^q} \d z \d t\rp ^{\frac{p}{q}} \d y \\
			&\quad  \quad+  \frac{(q(1-s))^{\frac{p}{q}}}{|B(x,r)\cap \Om|^{\frac{p}{q}}}\Ia{y\in B(x,r)\cap\Om}\lp \frac{1}{r^{q-sq}}\I{t=0}^r t^{q-sq-1}  \I{z\in B(x,r)\cap\Om\setminus B(y,t)} \frac{|f(y)-f(z)|^q}{|y-z|^q} \d z \d t  \rp ^{\frac{p}{q}} \d y \\
			&\quad = I_1+I_2.
		\end{align*}
		We evaluate $I_1$ and $I_2$ separately below. For $I_1$, we use Fubini's theorem and measure density of $\Om$:
		\begin{align*}
			I_1&\leq 
			\frac{(q(1-s))^{\frac{p}{q}}}{|B(x,r)\cap \Om|^{\frac{p}{q}}}\Ia{y\in B(x,r)\cap\Om} \lp \frac{r^n}{r^{q-sq}}\I{t=0}^r t^{q-sq-n-1}  \I{z\in B(y,t)\cap \Om} \frac{|f(y)-f(z)|^q}{|y-z|^q} \d z \d t\rp ^{\frac{p}{q}} \d y\\
			&\leq \frac{(q(1-s))^{\frac{p}{q}}}{|B(x,r)\cap \Om|^{\frac{p}{q}}}\Ia{y\in B(x,r)\cap\Om} \lp \frac{r^n}{r^{q-sq}}\I{z\in B(x,r)\cap\Om} \I{t=|y-z|}^{r}t^{q-sq-n-1}\frac{|f(y)-f(z)|^q}{|y-z|^q} \d t \d z\rp ^{\frac{p}{q}} \d y\\
			&\leq \frac{(q(1-s))^{\frac{p}{q}}}{|B(x,r)\cap \Om|^{\frac{p}{q}}}\Ia{y\in B(x,r)\cap\Om} \lp \frac{r^n}{r^{q-sq}}\I{z\in B(x,r)\cap\Om} \I{t=|y-z|}^{\infty}t^{q-sq-n-1}\frac{|f(y)-f(z)|^q}{|y-z|^q} \d t \d z\rp ^{\frac{p}{q}} \d y\\
			&= \frac{(q(1-s))^{\frac{p}{q}}r^{(\frac{np}{q}+ sp-p)}}{(n-q+sq)^{\frac{p}{q}} |B(x,r)\cap\Om|^{\frac{p}{q}}}  \Ia{y\in B(x,r)\cap\Om} \lp \I{z\in B(x,r)\cap\Om} \frac{|f(y)-f(z)|^q}{|y-z|^{n+sq}} \d z\rp ^{\frac{p}{q}} \d y\\
			&= \frac{(q(1-s))^{\frac{p}{q}}r^{( sp-p)}}{(n-q+sq)^\frac{p}{q}}  \Ia{y\in B(x,r)\cap\Om} \lp \I{z\in B(x,r)\cap\Om} \frac{|f(y)-f(z)|^q}{|y-z|^{n+sq}} \d z \rp ^{\frac{p}{q}}\d y
		\end{align*}
		For $I_2$, we have
		\begin{align*}
			I_2&= \frac{(q(1-s))^{\frac{p}{q}}}{|B(x,r)\cap \Om|^{\frac{p}{q}}}\Ia{y\in B(x,r)\cap\Om}\lp \frac{1}{r^{q-sq}}\I{t=0}^r t^{q-sq-1}  \I{z\in B(x,r)\cap\Om\setminus B(y,t)} \frac{|f(y)-f(z)|^q}{|y-z|^q} \d z \d t  \rp ^{\frac{p}{q}} \d y\\
			&\leq  \frac{(q(1-s))^{\frac{p}{q}}}{|B(x,r)\cap \Om|^{\frac{p}{q}}} \Ia{y\in B(x,r)\cap\Om}\lp  \frac{1}{r^{q-sq}}\I{t=0}^r t^{p-sp-1} \I{z\in [B(x,r)\cap\Om] \setminus B(y,t)} \frac{|f(y)-f(z)|^q}{|y-z|^{n+sq}} t^{n+sq-q} \d t\d z \rp ^{\frac{p}{q}} \d y\\
			&\leq  \frac{(q(1-s))^{\frac{p}{q}}}{|B(x,r)\cap \Om|^{\frac{p}{q}}} \Ia{y\in B(x,r)\cap\Om}\lp  \frac{1}{r^{q-sq}}\I{t=0}^r t^{n-1}\d t \I{z\in [B(x,r)\cap\Om] \setminus B(y,t)} \frac{|f(y)-f(z)|^q}{|y-z|^{n+sq}} \d z \rp ^{\frac{p}{q}} \d y\\
			&= \frac{(q(1-s))^{\frac{p}{q}}r^{( sp-p)}}{(n-q+sq)^{\frac{p}{q}}}  \Ia{y\in B(x,r)\cap\Om} \lp \I{z\in B(x,r)\cap\Om} \frac{|f(y)-f(z)|^q}{|y-z|^{n+sq}} \d z \rp ^{\frac{p}{q}}\d y
		\end{align*}
		Adding $I_1$ and $I_2$, from our initial estimate, we get
		
		\begin{equation}\label{49}
			\begin{split}
				&\Ia{y\in B(x,r)\cap\Om}\lp \Ia{z\in B(x,r)\cap\Om} \frac{|f(y)-f(z)|^q}{|y-z|^q} \d z \rp ^{\frac{p}{q}}\d y\\	
				& \qquad \qquad \leq \frac{(q(1-s))^{\frac{p}{q}}r^{( sp-p)}}{(n-q+sq)^\frac{p}{q}} \Ia{y\in B(x,r)\cap\Om} \lp \I{B(x,r)\cap\Om} \frac{|f(y)-f(z)|^q}{|y-z|^{n+sq}}  \d z\rp ^{\frac{p}{q}} \d y.
			\end{split}  
		\end{equation}

		For $y,z\in B(x,r)\cap\Om$ implies $|y-z|<2r$. So, 
		\begin{equation}\label{50}
			\begin{split}
				(2r)^{-p}\Ia{y\in B(x,r)\cap\Om}  &\abs{f(y)-f_{B(x,r)\cap\Om}}^p \d y\\
				&\leq\Ia{y\in B(x,r)\cap\Om} \Ia{z\in B(x,r)\cap\Om} \frac{|f(y)-f(z)|^p}{(2r)^p} \d z \d y\\
				&\leq \Ia{y\in B(x,r)\cap\Om} \Ia{B(x,r)\cap\Om} \frac{|f(y)-f(z)|^p}{|y-z|^{p}}  \d z \d y\\
				& \leq \Ia{y\in B(x,r)\cap\Om} \lp \Ia{B(x,r)\cap\Om} \frac{|f(y)-f(z)|^q}{|y-z|^{q}}  \d z \rp ^{\frac{p}{q}}\d y
			\end{split}
		\end{equation}
		By using \cref{49}, from \cref{50} we get that 
		\begin{equation*}
			\begin{split}
				\Ia{y\in B(x,r)\cap\Om}  &\abs{f(y)-f_{B(x,r)\cap\Om}}^p \d y\\
				&\leq \frac{(q(1-s))^{\frac{p}{q}}r^{sp}}{2^p(n-q+sq)^{\frac{p}{q}}}\Ia{y\in B(x,r)\cap\Om} \lp \I{B(x,r)\cap\Om} \frac{|f(y)-f(z)|^q}{|y-z|^{n+sq}}  \d z\rp ^{\frac{p}{q}} \d y.  
			\end{split}
		\end{equation*}
	\end{proof}

	The following lemma establishes the fractional extension estimates. Our approach utilises the refined estimates from \Cref{Ponce1} to achieve the sharp $s$-dependence as $s\to 1$, which is essential for our main characterisation. This provides a more precise control over the constants than what is available in \cite[Theorem 1.1]{Zhou2015}.

	\begin{lemma}\label{frac-extension} 
		Let $\Om \subseteq \BB:=B(0,R)$ be an open set with measure density condition, and $p\in (1,\infty)$. Then there exists $s_0\in(0,1)$ such that for any $s\in(s_0,1)$,  
		
		\begin{equation}\label{eq1}
			\I{\BB}\I{\BB}\frac{|Ef(x)-Ef(y)|^p}{|x-y|^{n+sp}}\d y \d x 
			\leq C(n,p,\Om,R) \I{\Om}\I{\Om}\frac{|f(x)-f(y)|^p}{|x-y|^{n+sp}} \d y \d x.
		\end{equation}

	\end{lemma} 
	\begin{proof} 
		Without loss of generality, we assume $R=4$, that is $\BB=B(0,4)$ as we can always move to the general case by a scaling. We also recall that an open set satisfying measure density condition must have a boundary of Lebesgue measure zero (see \cite[Lemma~3.5]{HKT-8}), so we need not be careful about $x\in \partial \Om$ when dealing with integrals. 
		
		Consider $\mathcal{W}=\mathcal{W}(\RR^n\setminus \Ombar )$. 
		For each $Q\in \mathcal{W}$, let $x_Q^\star$ be as in \Cref{lem:whitney}. We define
		$$ B_{Q}^\star:= B(x_Q^\star,\diam(Q))\cap \Om.$$
		For $x\in\BB\setminus\Ombar $, denote by $\mathcal{W}_x$ the collection of all $Q\in \mathcal{W}$ such that $x\in Q$. Then from \Cref{lem:whitney}, $\#\mathcal{W}_x\leq M$ for some $M\in \NN$. 
		
		Note that, if $Q\in \mathcal{W}$ be such that $\diam(Q)\geq 1$, then, for $z\in 2Q, d(z,\Om) \geq 8\diam(Q)> 8$, which is not possible since $\BB=B(0,4)$. Hence, we have
		$$
		\diam(Q)< 1 \quad \text{ for all } Q\in \mathcal{W}.
		$$
		Also, for $Q\in \mathcal{W}_x$, by \Cref{lem:whitney}, 
		\begin{equation}\label{eq11}
			B(x_Q^\star,\diam(Q)) \subseteq 20Q,
		\end{equation}
		and consequently $|B_Q^\star| \approx |Q|$. Now \cref{measure-density} gives
		\begin{equation}\label{MEQ}
			|B_Q^\star| \approx |20Q \cap \Om| \approx |20Q|.
		\end{equation}

		For $f\in L^{1}_{\loc}(\Om)$, we define $E$ by setting
		\begin{equation}\label{def1}
			Ef(x):= \begin{cases}
				\sum_{Q\in \mathcal{W}}f_{B_Q^\star} \varphi_Q(x), &\ {\rm if}\ x\in \BB\setminus\Ombar ,\\
				f(x),&\ {\rm if}\ x\in \Om.
				
			\end{cases}
		\end{equation}
		We write 
		\begin{equation}\label{eq14}
			\begin{split}
				\I{\BB}\I{\BB}\frac{|Ef(x)-Ef(y)|^p}{|x-y|^{n+sp}}\d y \d x
				&\leq  \I{\Om}\I{\Om}\frac{|f(x)-f(y)|^p}{|x-y|^{n+sp}} \ \d y \d x \\
				& \quad  + 2 \I{\BB\setminus \Om} \I{\Om} \frac{|Ef(x)-Ef(y)|^p}{|x-y|^{n+sp}}\d y \d x \\
				& \quad +\I{\BB\setminus \Om} \I{\BB \setminus\Om} \frac{|Ef(x)-Ef(y)|^p}{|x-y|^{n+sp}}\d y \d x \\
				& =  I_1 + 2 I_2 + I_3 .
			\end{split}
		\end{equation} 
		We will estimate $I_2 $ and $I_3$ separately below.

		\textbf{Estimate for $I_2$:}
		Let  $x\in \BB \setminus \Om$ and $y \in \Om$. Fix $Q_0\in \mathcal{W}_x$. From \cref{def1}, we get
		\begin{equation}\label{I_2}
			\begin{split}
				|Ef(x)- f(y)| &\leq \sum_{Q\in \mathcal{W}} \varphi_Q(x)|f_{B_Q^\star}-f(y)|\\
				&\leq \sum_{Q\in \mathcal{W}} \varphi_Q(x)\Ia{B_Q^\star}|f(z)-f(y)| \d z.
			\end{split}
		\end{equation}
		Combining \cref{I_2,eq11,MEQ} with \Cref{lem:whitney,POU}, we get
		\begin{equation}\label{eq3.6}
			|Ef(x)- f(y)|
			\leq  C(n) \Ia{20Q_0\cap \Om}|f(z)-f(y)| \d z.
		\end{equation}

		For $y\in \Om$, $z\in 20Q_0\cap \Om $, since $|x-y| \geq d(x,\Om)$, we always have
		\begin{equation*}
			|z-y|\lesssim |x-y|.
		\end{equation*}
		Hence \cref{eq3.6} gives
		\begin{equation}\label{neeq}
			\begin{split}
				\frac{ |Ef(x)- f(y)|}{|x-y|^{\frac{n}{p}+s}} &\leq C(n,p) \Ia{20Q_0\cap \Om } \frac{ |f(z)- f(y)|}{|z-y|^{\frac{n}{p}+s}} \d z\\
				& \leq C(n,p)\Ia{20Q_0 } \frac{ |f(z)- f(y)|}{|z-y|^{\frac{n}{p}+s}}\chi_{\Om}(z) \d z\\
				& \leq C(n,p)\mathcal{ M}\lp \frac{ |f(\cdot)- f(y)|}{|(\cdot)-y|^{\frac{n}{p}+s}}\chi_{\Om}(\cdot)\rp (x).
			\end{split}
		\end{equation}
		
		Raising both sides to power $p$ and integrating, and then using \Cref{HLMI} we get, for any $y\in \Om$,
		\begin{equation}\label{eq3.9}
			\begin{split}
				I_2=\I{\BB\setminus \Om} \frac{|Ef(x)-Ef(y)|^p}{|x-y|^{n+sp}}\d x 
				& \leqa \I{\RR^n} \lb  \mathcal{M} \lp \frac{|f(\cdot)-f(y)|}{|(\cdot)-y|^{\frac{n}{p}+s}}\chi_{\Om}(\cdot)\rp (x)\rb^p\d x \\
				& \leqa \I{\Om} \frac{|f(x)-f(y)|^p}{|x-y|^{n+sp}} \d x .
			\end{split}
		\end{equation}
		Integrating again over $\Om$,
		\begin{equation}\label{eq13}
			I_2
			\leqa  \I{\Om}\I{\Om}\frac{|f(x)-f(y)|^p}{|x-y|^{n+sp}}\d x \d y.
		\end{equation}

		\textbf{Estimate for $I_3$:} For $x\in \BB \setminus\Om$, fix $Q_0\in \mathcal{W}$ arbitrarily so that $Q_0 \in \mathcal{W}_x$. We split the region $\BB\setminus \Om$ into $X_1(x)$ and $X_2(x)$ respectively. Set 
		$$ X_1(x):= \{y\in \BB \setminus \Om, |x-y|\geq \frac{1}{2} \max\{d(x,\Om),d(y,\Om)\}\}$$
		and
		$$ X_2(x):= \{y\in \BB \setminus \Om, |x-y|<\frac{1}{2} \max\{d(x,\Om),d(y,\Om)\}\}.$$
		We write
		\begin{equation}\label{eq15}
			\begin{split}
				I_3=\I{\BB \setminus \Om}\I{\BB \setminus \Om}\frac{|Ef(x)-Ef(y)|^p}{|x-y|^{n+sp}}\d y \d x &=  \I{\BB \setminus \Om}\I{X_1(x)}\frac{|Ef(x)-Ef(y)|^p}{|x-y|^{n+sp}}\d y \d x\\
				&+\I{\BB \setminus \Om}\I{X_2(x)}\frac{|Ef(x)-Ef(y)|^p}{|x-y|^{n+sp}}\d y \d x.
			\end{split}     
		\end{equation}
		For a $y\in X_{1}(x)$ and $Q_1 \in \mathcal{W}_y$, we have
		\begin{equation}\label{eq3.13}
			\begin{split}
				Ef(x)-Ef(y) &=  \sum_{Q \in \mathcal{W}_x}f_{B_Q^\star}\varphi_Q (x)- \sum_{R\in \mathcal{W}_y}f_{B_R^\star}\varphi_R(y)\\
				& =  \sum_{R\in \mathcal{W}_y} \varphi_R(y) \sum_{Q\in \mathcal{W}_x}f_{B_Q^\star}\varphi_Q(x) -  \sum_{Q\in \mathcal{W}_x}\varphi_Q(x)\sum_{R\in \mathcal{W}_y}f_{B_R^\star}\varphi_R(y)\\
				& =  \sum_{R\in \mathcal{W}_y}  \sum_{Q\in \mathcal{W}_x} \varphi_Q(x)\varphi_R(y)(f_{B_Q^\star}-f_{B_R^\star}).
			\end{split}
		\end{equation}
		Now, \cref{MEQ,eq11} imply
		\begin{equation}\label{eq3.15}
			\begin{split}
				| (f_{B_Q^\star}-f_{B_R^\star})| 
				& \leq \Ia{B_Q^\star} \Ia{B_R^\star} |f(z)-f(w)|\d z\ dw\\
				&\leq (C(n))^2 \Ia{20Q_1\cap \Om} \Ia{20Q_0\cap \Om} |f(z)-f(w)|\d z\ dw.
			\end{split}
		\end{equation}
		
		For $z \in 20Q_1\cap \Om$ and $w\in  20Q_0\cap \Om $, since  $|x-y| \geq \frac{1}{2}  \max\{d(x,\Om),d(y,\Om)\}\} $, we get that
		$$|z-w|\lesssim |x-y|.$$
		We use this along with \cref{eq3.13,eq3.15} to conclude that
		\begin{equation}\label{eq3.16}
			\begin{split}
				\frac{ |Ef(x)- Ef(y)|}{|x-y|^{\frac{n}{p}+s}} &\leq C(n,p) \Ia{20Q_1\cap \Om } \Ia{20Q_0\cap \Om }\frac{ |f(z)- f(w)|}{|z-w|^{\frac{n}{p}+s}} \d z\ dw\\
				& \leq C(n,p) (\mathcal{M} \times \mathcal{M})(G)(x,y)
			\end{split}
		\end{equation}
		where
		$$ G(z,w)= \frac{|f(z)-f(w)|}{|w-z|^{\frac{n}{p}+s}}\chi_{\Om}(z)\chi_{\Om}(w).$$
		By raising the power $p$ over \cref{eq3.16} and then applying \Cref{HLMI}, we conclude that
		\begin{equation}\label{eq16}
			\begin{split}
				\I{\BB \setminus \Om}\I{X_1(x)}\frac{|Ef(x)-Ef(y)|^p}{|x-y|^{n+sp}}\d y \d x  &\leqa \I{\RR^n}\I{ \RR^n} [(\mathcal{M} \times \mathcal{M})(G)(x,y)]^p  \d x \d y\\
				&\leqa 
				\I{\Om} \I{\Om} \frac{ |f(x)- f(y)|^p}{|x-y|^{n+sp}} \d x \d y.
			\end{split}
		\end{equation}
		Now, for the case $y \in X_2(x)$, notice that $\sum_{Q \in \mathcal{W}_x\cup \mathcal{W}_y}[\varphi_Q(x)-\varphi_Q(y)]=0$.
		Hence,
		\begin{equation*}
			\begin{split}
				|Ef(x)-Ef(y)| 
				&= \abs{\sum_{Q\in \mathcal{W}_x\cup \mathcal{W}_y}[\varphi_Q(x)-\varphi_Q(y)]f_{B_Q^\star}}\\
				& = \abs{\sum_{Q\in \mathcal{W}_x\cup \mathcal{W}_y}[\varphi_Q(x)-\varphi_Q(y)](f_{B_Q^\star}-f_{20Q_0\cap \Om})}\\
				& \leqa \sum_{Q\in \mathcal{W}_x\cup \mathcal{W}_y}\frac{|x-y|}{\diam{(Q)}}|f_{B_Q^\star}-f_{20Q_0\cap \Om}|\\
				& \leqa \sum_{Q\in \mathcal{W}_x\cup \mathcal{W}_y}\frac{|x-y|}{\diam{(Q)}}\Ia{20Q_0\cap \Om}|f(z)-f_{20Q_0\cap \Om}|\d z.
			\end{split}
		\end{equation*}
		By raising the power $p$ and using \Cref{lem:whitney} and \Cref{POU}, we get
		\begin{equation*}
			\begin{split}
				|Ef(x)-Ef(y)|^p 
				&\leqa \sum_{Q \in \mathcal{W}_x\cup \mathcal{W}_y}\frac{|x-y|^p}{\diam{(Q)^p}}\lp \Ia{20Q_0\cap \Om}|f(z)-f_{20Q_0\cap \Om}|\d z\rp ^p\\
				&\leqa \frac{|x-y|^p}{\diam{(Q)^p}}\lp \Ia{20Q_0\cap \Om}|f(z)-f_{20Q_0\cap \Om}|\d z\rp ^p.
			\end{split}
		\end{equation*}
		For all $Q\in \mathcal{W}_x\cup \mathcal{W}_y$, since $\diam(Q) \approx d(x,\Om)$, we get
		\begin{equation*}
			|Ef(x)-Ef(y)|^p 
			\leqa \frac{|x-y|^p}{d(x,\Om)^p}\lp \Ia{20Q_0\cap \Om}|f(z)-f_{20Q_0\cap \Om}|\d z\rp ^p.
		\end{equation*}
		Now, using \cref{measure-density}, Jensen's inequality and \Cref{Ponce1} (with $1$ in place of $p$, and $p$ in place of $q$), we get that 
		\begin{equation*}
			\begin{split}
				\I{Q_0}&\I{X_2(x)}\frac{|Ef(x)-Ef(y)|^p}{|x-y|^{n+sp}}\d y \d x \\
				& \leqa  \I{Q_0}\I{X_2(x)} \frac{|x-y|^{p-n-sp}}{d(x,\Om)^p}\lp \Ia{20Q_0\cap \Om}|f(z)-f_{20Q_0\cap \Om}|\d z\rp ^p\d y \d x\\
				&\leqa (1-s)\I{Q_0}\I{X_2}\frac{|x-y|^{p-n-sp}}{d(x,\Om)^{p-sp}}\lb \Ia{20Q_0\cap\Om}\lp \I{20Q_0\cap\Om}\frac{|f(z)-f(w)|^p}{|z-w|^{n+sp}}\d z\rp ^{\frac{1}{p}}\ dw\rb^p \d y \d x\\
				&\leqa (1-s)\I{Q_0}\I{X_2}\frac{|x-y|^{p-n-sp}}{d(x,\Om)^{p-sp}}\d y \lb \Ia{20Q_0\cap\Om}\lp \I{20 Q_0\cap\Om}\frac{|f(z)-f(w)|^p}{|z-w|^{n+sp}}\d z\rp ^{\frac{1}{p}}\ dw\rb^p \d x \\
				&\leqa \I{Q_0}\lb \Ia{20Q_0\cap\Om}\lp \I{20Q_0\cap\Om}\frac{|f(z)-f(w)|^p}{|z-w|^{n+sp}}\d z\rp ^{\frac{1}{p}}\ dw\rb^p \d x \\
				&\leqa \I{Q_0}\lb \mathcal{M}\lp \lp \I{\Om}\frac{|f(z)-f(\cdot)|^p}{|z- \cdot|^{n+sp}}\d z\rp ^{\frac{1}{p}}\chi_{\Om}(\cdot)\rp (x)\rb^p \d x .
			\end{split}   
		\end{equation*}
		Since $Q_0$ is an arbitrary cube in $\mathcal{W}$, summing over $Q_0\in \mathcal{W}$ and applying \Cref{lem:whitney}, we get that
		\begin{equation}\label{33}
			\begin{split}
				\I{\BB \setminus \Om}&\I{X_2(x)}\frac{|Ef(x)-Ef(y)|^p}{|x-y|^{n+sp}}\d y \d x\\
				&=\sum_{Q_0\in \mathcal{W}}\I{Q_0}\I{X_2(x)}\frac{|Ef(x)-Ef(y)|^p}{|x-y|^{n+sp}}\d y \d x\\
				&\leqa \sum_{Q_0\in \mathcal{W}}\I{Q_0}\lb \mathcal{M}\lp \lp \I{\Om}\frac{|f(z)-f(\cdot)|^p}{|z-w|^{n+sp}}\d z\rp ^{\frac{1}{p}}\chi_{\Om}(\cdot)\rp (x)\rb^p \d x \\
				&\leqa \I{\RR^n}\lb \mathcal{M}\lp \lp \I{\Om}\frac{|f(z)-f(\cdot)|^p}{|z-w|^{n+sp}}\d z\rp ^{\frac{1}{p}}\chi_{\Om}(\cdot)\rp (x)\rb^p \d x .
			\end{split}
		\end{equation}
		\Cref{HLMI} and \cref{33} implies
		\begin{equation}\label{eq12}
			\begin{split}
				\I{\BB \setminus \Om}\I{X_2(x)}\frac{|Ef(x)-Ef(y)|^p}{|x-y|^{n+sp}}\d y \d x
				&\leqa \I{\Om}\I{\Om}\frac{|f(x)-f(y)|^p}{|x-y|^{n+sp}}\d y \d x.
			\end{split}
		\end{equation}
		The proof follows from \cref{eq12,eq16,eq15,eq13,eq14}.
	\end{proof} 
	
	Now, can start the proof of our first main theorem.
	
	\begin{proof}[\textbf{Proof of \Cref{th-main}}]

		\textbf{(1 $\implies$ 2):} Assume (1) holds. Then \cref{bbm-hom} can be proved similarly to \cite{BoBrMi}, using \cite[Corollary 4.9]{Heronkoskela}, which essentially says that nonhomogeneous and homogeneous extension domains are the same, provided that the domain is bounded. Finally, the measure density condition, i.e. \cref{measure-density}, directly follows from \cite[Theorem 6]{HKT}.
		
		\textbf{(2 $\implies$ 3):} This is trivial as $W^{1,p}(\Om)\subseteq  \Wdot{1,p}(\Om)$.
		
		\textbf{(3 $\implies$ 1):} Assume that (3) holds. Without loss of generality, assume $\Om\subseteq  \Ombar  \subseteq \BB=B(0,4)$. It is well-known that $\BB$ is an $(1,p)$-extension domain. Hence, in order to produce a bounded extension operator $\mathcal{E}:W^{1,p}(\Om)\to W^{1,p}(\RR^n)$, it is enough to produce an extension operator $E:W^{1,p}(\Om)\to W^{1,p}(\BB)$.
		
		Let $f\in W^{1,p} (\Om)$. Since \cref{measure-density} holds, we can apply \Cref{frac-extension} to get $Ef:\BB\to \RR$ with $Ef|_{\Om}=f$ and $s_0\in(0,1)$ such that \cref{eq1} holds. Combining \cref{bbm} with \cref{eq1}, we get, for any $s\in (s_0,1)$, $$ 
		(1-s) \I{\BB}\I{\BB}\frac{|Ef(x)-Ef(y)|^p}{|x-y|^{n+sp}}\d y \d x 
		\leqa \lp  \I{\Om} |\nabla f|^p
		+ \I{\Om} |f|^p\rp . 
		$$ 
		Now, we take limit as $s\to 1-$ in this inequality and then apply \Cref{bbm-formula} to conclude that 
		$$ 
		\I{\BB} |\nabla Ef|^p \leq C(n,p,\Om) \lp \I{\Om} |\nabla f|^p +\I{\Om}|f|^p\rp .  $$ 
		To complete the proof, we need to show that $Ef\in L^p(\BB)$.
		Let $v$ is the zero-extension of $|f|$ to $\RR^n$. For $x\in \BB\setminus \Om$. From definition \cref{def1} and \Cref{lem:whitney}, we have
		\begin{equation*}\label{mx}
			\begin{split}
				|Ef(x)|
				&=\abs{\sum_{Q\in \mathcal{W}_x}(f)_{B_Q^\star}\varphi_Q(x)}\\
				&\leqa   \sum_{Q\in \mathcal{W}_x}  \Ia{20Q} |v(y)|\d y \\
				&\leqa \mathcal{M}(v)(x).
			\end{split}     
		\end{equation*}
		Hence,
		\begin{equation*}
			\|Ef\|_{L^p(\BB\setminus \Om)}^p
			\leqa \I{\BB\setminus\Om} |\mathcal{M}(v)(x)|^p\d x \leqa \I{\RR^n}  |\mathcal{M}(v)(x)|^p\d x 
		\end{equation*}
		Using \Cref{HLMI}, we get that
		\begin{equation*}
			\|Ef\|_{L^p(\BB\setminus \Om)}^p
			\leqa \I{\Om}|f(x)|^p\d x 
		\end{equation*}
		Hence
		\begin{equation*}\label{mx1}
			\|Ef\|_{L^p(\BB)}^p \leqa \I{\Om}|f(x)|^p\d x 
		\end{equation*}
		
		So, we have proven equivalence of (1), (2), and (3).
		
		\textbf{(4 $\implies$ 5):} Trivial.
		
		\textbf{(1 $\implies$ 4):} Follows from the same proof as in \cite{BoBrMi}, with the help of \cite[Corollary 4.9]{Heronkoskela}.

		\textbf{(5 $\implies$ 1):} As above, we establish the existence of an extension operator $E:W^{1,p}(\Om)\to W^{1,p}(\BB)$. We show that the Whitney operator defined in \Cref{frac-extension} does the job for us.
		
		Let $f\in W^{1,p} (\Om)$. Since \cref{measure-density} holds, we get from \Cref{frac-extension},
		$$
		(1-s)\I{\BB}\I{\BB}\frac{|Ef(x)-Ef(y)|^p}{|x-y|^{n+sp}}\d y \d x 
		\leq C(n,p,\Om) (1-s) \I{\Om}\I{\Om}\frac{|f(x)-f(y)|^p}{|x-y|^{n+sp}} \d y \d x.
		$$
		Taking limit as $s\to1-$ on both sides, we get by the hypothesis,
		$$
		\limsup\limits_{s\to 1-} (1-s)\I{\BB}\I{\BB}\frac{|Ef(x)-Ef(y)|^p}{|x-y|^{n+sp}}\d y \d x 
		\leq C(n,p,\Om) K(n,p) \I{\Om} \abs{\nabla f}^p \d x <\infty.
		$$
		Now we can apply \Cref{bbm-formula} to conclude that $Ef\in \W{1,p}(\BB)$.
		
		This completes the proof.
	\end{proof}

	\vspace{1cm}
	
	\section{BBM domains which are not $(1,p)$-extension domains}\label{sec-examples}
	
	We want to provide a class of examples of open sets supporting the BBM inequality beyond the class of domains with measure density condition. We do it in our next result using the notion of Haj\l asz Sobolev spaces (see \Cref{preli} for definitions).	
	\begin{proposition}\label{bbmeq}
		Let $1<p<\infty$, $s\in(0,1)$, and $\Om\subseteq \RR^n$ be a bounded domain such that $\Mdot{1,p}(\Om)= \Wdot{1,p}(\Om)$. Then, for all $f\in \Wdot{1,p}(\Om)$,  
		\begin{equation*}
			(1-s)\I{\Om}\I{\Om}\frac{|f(x)-f(y)|^p}{|x-y|^{n+sp}}\d y \d x
			\leq C(n,p, \Om) \I{\Om} |\nabla f|^p.
		\end{equation*}
	\end{proposition}

	\begin{proof}
		Let $f\in \Wdot{1,p}(\Om)$. Since $\Mdot{1,p}(\Om)= \Wdot{1,p}(\Om)$, there exist $g\in L^p(\Om)$ such that
		\begin{equation*}
			|f(x)-f(y)| \leq |x-y|(g(x)+g(y)).
		\end{equation*}
		Hence,
		\begin{equation*}
			\begin{split}
				(1-s)\I{\Om}\I{\Om}\frac{|f(x)-f(y)|^p}{|x-y|^{n+sp}}\d y \d x & \leq 2^{p-1}(1-s)\I{\Om}\I{\Om}  \frac{g(x)^p+ g(y)^p}{|x-y|^{n+sp-p}}\d y \d x\\ 
				&\leq 2^p (1-s)\I{\Om}\I{\Om}\frac{g(x)^p}{|x-y|^{n+sp-p}}\d y \d x.
			\end{split}
		\end{equation*}
		
		For a fixed $x\in \Om$, consider the change of variables $z=x-y$. We get
		\begin{equation}\label{bbmf}
			\begin{split}
				(1-s)\I{\Om}\I{\Om}\frac{|f(x)-f(y)|^p}{|x-y|^{n+sp}}\d y \d x 
				&\leq 2^p (1-s) \I{\Om}\I{B(0,2\diam(\Om))}\frac{g(x)^p}{|z|^{n+sp-p}}\d z\d x \\
				&\leq C(n,p, \Om) \I{\Om} g(x)^p\d x .
			\end{split}
		\end{equation}
		Since $\Mdot{1,p}(\Om)= \Wdot{1,p}(\Om)$, by using \cref{bbmf}  the proof follows.
	\end{proof}
	The above result allows us to construct specific examples of domains where the BBM inequality holds despite the failure of the $(1,p)$-extension property.
	\begin{example}\label{cusp-example}
		For $\lambda>1$, define 
		$$
		\Om_\lambda:= \lc  (x_1,x')\ \big|\ x'\in \RR^{n-1},\ x_1\in(0,1),\ |x'|< x_1^\lambda \ \rc .
		$$
		A close reading of the proof of \cite[Theorem 1.1]{EKMZ} shows that for $\Om_\lambda$, we have $\Mdot{1,p}(\Om_\lambda)= \Wdot{1,p}(\Om_\lambda)$. Although the statement in \cite[Theorem 1.1]{EKMZ} is for the nonhomogeneous case, the same proof works for the homogeneous case as well. Hence, by \Cref{bbmeq}, $\Om_\lambda$ supports the BBM inequality, but since it does not satisfy measure density condition, it cannot be a $(1,p)$-extension domain by \cite{HKT-8}. 
	\end{example}

	\vspace{1cm} 	 	\section{Self-improvement of extension}	\label{se}

	Throughout this section, we shall assume that $\Om \subseteq \bar{\Om} \subseteq B(0,1)$. Note that, even if we prove \Cref{th-self-improve} with this additional restriction, the general result follows immediately by a scaling argument. Consider $\mathcal{W}=\mathcal{W}(\RR^n\setminus \Ombar )$. For each $Q \in \mathcal{W}$, we define $\Tilde{Q}$ to be a cube having the same centre with $\diam (\Tilde{Q})= \diam(Q)^{\frac{1}{s}}$.

	\begin{lemma}\label{th1}
		Let $\Om\subseteq \RR^n$ be an open set with $|\partial \Om|=0$. Let $p \in (1,\infty)$, and $ s \in(0,1)$. Assume further that $E_1: \Wdot{1,p}(\Om)\to \Wdot{s,p}(\RR^n)$ is a bounded extension operator. Define another extension operator $E_2$ by setting        \begin{equation}\label{eqC.13}
			E_2f(x):= \begin{cases}
				\sum_{Q\in \mathcal{W}}(E_1f)_{\Tilde{Q}}\varphi_Q(x), &\ {\rm if}\ x\in B(0,1)\setminus\Ombar ,\\
				f(x),&\ {\rm if}\ x\in \Om,\\
				E_1f(x), &\ {\rm if}\ x \in \partial\Om.
			\end{cases}
		\end{equation}
		If
		$$
		E_2 \lp C^\infty(\Om)\cap \Wdot{1,p}(\Om)\rp  \subseteq \Wdot{1,1}_{\loc}(B(0,1)),
		$$
		then  $E_2:  \Wdot{1,p}(\Om)\to \Wdot{1,p}(B(0,1))$ is bounded.
	\end{lemma}
	\begin{proof}
		We use the hypothesis
		$$
		E_2 \lp C^\infty(\Om)\cap \Wdot{1,p}(\Om)\rp  \subseteq \Wdot{1,1}_{\loc}(B(0,1)),
		$$
		to conclude that $E_2f$ has a weak gradient $\nabla E_2f$, which is defined in $B(0,1)$. Thus, we only need to show the seminorm inequality.
		
		Let $x\in \BB\setminus \bar{\Om}$ and  $Q_0\in \mathcal{W}_x$. Then we have
		\begin{equation*}
			\begin{split}
				\nabla E_2 f (x)& = \nabla \big(\sum_{Q\in \mathcal{W}_x} (E_1f)_{\Tilde{Q}}\varphi_Q(x)\big)\\
				& = \nabla\big( \sum_{Q\in \mathcal{W}_x}\big((E_1f)_{\Tilde{Q}} -(E_1f)_{\Tilde{Q}_0} \big)\varphi_Q(x)\big)
			\end{split}
		\end{equation*}
		From \Cref{POU}, we know, for $Q\in \mathcal{W}_x$,  $|\nabla \varphi_Q(x)|\approx \diam(Q)^{-1}$. Hence,
		$$
		|\nabla E_2 f(x)|^p\leqa \lb \sum_{Q\in \mathcal{W}_x}\frac{1}{\diam(Q)}\Ia{\Tilde{Q}}\Ia{\Tilde{Q}_0}|(E_1f)(y)- (E_1f)(z)|\d z\d y \rb^p.
		$$
		By H\"older inequality, we get that
		\begin{equation*}
			\begin{split}
				|\nabla E_2 f(x)|^p
				&  \leqa \lb \sum_{Q\in \mathcal{W}_x}\frac{1}{\diam(Q)}\Ia{\Tilde{Q}}\lp \Ia{\Tilde{Q}_0}|(E_1f)(y)- (E_1f)(z)|^p\d z\rp ^\frac{1}{p} \d y \rb^p\\
				&  \leqa \lb \sum_{Q\in \mathcal{W}_x}\frac{1}{\diam(Q)}\Ia{\Tilde{Q}}\lp \I{\Tilde{Q}_0}\frac{|(E_1f)(y)- (E_1f)(z)|^p}{|y-z|^n}\d z\rp ^\frac{1}{p} \d y \rb^p.
			\end{split}
		\end{equation*}
		This implies
		$$|\nabla E_2 f(x)|^p
		\leqa \lb \sum_{Q\in \mathcal{W}_x}\frac{\diam(\tilde{Q}_0)^s}{\diam(Q)}\Ia{\Tilde{Q}}\lp \I{\Tilde{Q}_0}\frac{|(E_1f)(y)- (E_1f)(z)|^p}{|y-z|^{n+sp}}\d z\rp ^\frac{1}{p} \d y \rb^p.
		$$
		Now $\diam(\tilde{Q})= \diam(Q)^\frac{1}{s}$ and, since $Q, Q_0\in \mathcal{W}_x$, we have that $ \diam(Q_0) \approx \diam(Q)$ and that there exist a $C>0$ such that $Q\subseteq C Q_0 $. Thus, the above inequality implies
		$$
		|\nabla E_2 f(x)|^p
		\leqa \lb \sum_{Q\in \mathcal{W}_x}\Ia{CQ_0}\lp \I{B(0,1)}\frac{|(E_1f)(y)- (E_1f)(z)|^p}{|y-z|^{n+sp}}\d z\rp ^\frac{1}{p} \d y \rb^p.
		$$
		We use (5) of \Cref{lem:whitney} and then take supremum over all cubes containing $x$,
		\begin{equation*}\label{fixed-s-ext-eq1}
			\begin{split}
				|\nabla E_2 f(x)|^p
				\leqa \lb \mathcal{M}\lp \I{B(0,1)}\frac{|(E_1f)(\cdot)- (E_1f)(z)|^p}{|\cdot-z|^{n+sp}}\d z\rp ^{\frac{1}{p}}(x)\rb^p.
			\end{split}
		\end{equation*}
		Now, by the boundedness of $E_1$ and \Cref{HLMI}, we get that
		\begin{equation*}
			\|\nabla E_2 f\|_{L^p(B(0,1)\setminus \bar{\Om})}\leqa  \|\nabla f(x)\|_{L^{p}(\Om)}.
		\end{equation*}
		The proof follows.
	\end{proof}

	\begin{lemma}\label{E2-in-fracsob}
		Let $p\in(1,\infty)$ and $s\in(0,1)$. Assume that $f\in C^\infty(\Om)\cap\dot W^{1,p}(\Om)$, $E_1f\in \dot W^{s,p}(B(0,1))$ is an extension of $f$, and that $E_2f$ is defined by \cref{eqC.13}. Then $E_2f\in \dot W^{s,p}(B(0,1))$.
	\end{lemma}
	
	\begin{proof}
		Set $\mathbb B:=B(0,1)$. Since $E_2f=E_1f=f$ on $\Om$, we write
		\[
		[E_2f]_{W^{s,p}(\mathbb B)}^p
		=\I{\BB}\I{\BB}
		\frac{|E_2f(x)-E_2f(y)|^p}{|x-y|^{n+sp}}\,dx\,dy
		=:I_{\Om\Om}+I_{\Om c}+I_{cc},
		\]
		where the three terms correspond to $(x,y)\in \Om\times\Om$, 
		$\Om\times(\mathbb B\setminus\Om)$, and
		$(\mathbb B\setminus\Om)\times(\mathbb B\setminus\Om)$.
		
		For $I_{\Om\Om}$, we simply have
		\[
		I_{\Om\Om}
		=\I{\Om}\I{\Om}
		\frac{|E_1f(x)-E_1f(y)|^p}{|x-y|^{n+sp}}\,dx\,dy
		\le [E_1f]_{W^{s,p}(\mathbb B)}^p.
		\]
		
		To estimate $I_{\Om c}$, fix $y\in\Om$ and $x\in\mathbb B\setminus\Om$.
		Using the definition of $E_2$, the Lipschitz bounds for the partition of
		unity $\{\varphi_Q\}_{Q\in\mathcal W}$, and the usual Whitney geometry, one
		obtains the standard pointwise estimate
		\[
		|E_2f(x)-f(y)|
		\le C\,\mathcal M\!\left(
		\frac{|E_1f(\cdot)-E_1f(y)|}{|\cdot-y|^{\frac np+s}}
		\chi_{\Om}(\cdot)
		\right)(x),
		\]
		where $\mathcal M$ is the Hardy--Littlewood maximal operator on $\mathbb R^n$
		(after extending the displayed function by $0$ outside $\Om$).
		Hence, by \Cref{HLMI},
		\[
		\int_{\mathbb B\setminus\Om}
		\frac{|E_2f(x)-f(y)|^p}{|x-y|^{n+sp}}\,dx
		\le
		C\int_{\Om}
		\frac{|E_1f(x)-E_1f(y)|^p}{|x-y|^{n+sp}}\,dx.
		\]
		Integrating in $y\in\Om$ yields
		\[
		I_{\Om c}\le C\,[E_1f]_{W^{s,p}(\mathbb B)}^p.
		\]
		
		It remains to estimate $I_{cc}$. Let $x,y\in\mathbb B\setminus\Om$.
		Write
		\[
		E_2f(x)=\sum_{Q\in\mathcal W_x}(E_1f)_{\widetilde Q}\varphi_Q(x),
		\qquad
		E_2f(y)=\sum_{R\in\mathcal W_y}(E_1f)_{\widetilde R}\varphi_R(y),
		\]
		and choose a reference cube $Q_x\in\mathcal W_x$.
		Set
		\[
		c_x:=(E_1f)_{C\widetilde Q_x},
		\qquad
		G(u):=\left(\int_{\mathbb B}\frac{|E_1f(u)-E_1f(v)|^p}{|u-v|^{n+sp}}\,dv\right)^{1/p},
		\quad u\in\mathbb B.
		\]
		Since $\operatorname{diam}(\widetilde Q)=\operatorname{diam}(Q)^{1/s}$, we have
		$\operatorname{diam}(\widetilde Q)^s=\operatorname{diam}(Q)$.
		Using cancellation $\sum_Q(\varphi_Q(x)-\varphi_Q(y))=0$, the Lipschitz bound
		\[
		|\varphi_Q(x)-\varphi_Q(y)|\lesssim \frac{|x-y|}{\operatorname{diam}(Q)},
		\]
		and a fractional Poincar\'e estimate on the enlarged Whitney cubes, we get
		\[
		|(E_1f)_{\widetilde Q}-c_x|
		\lesssim \operatorname{diam}(\widetilde Q)^s\,\mathcal M(G)(x)
		= \operatorname{diam}(Q)\,\mathcal M(G)(x).
		\]
		Therefore,
		\[
		|E_2f(x)-E_2f(y)|
		\lesssim |x-y|\bigl(\mathcal M(G)(x)+\mathcal M(G)(y)\bigr).
		\]
		Hence
		\[
		I_{cc}
		\lesssim
		\I{\BB\setminus\Om}\I{\BB\setminus\Om}
		\frac{\bigl(\mathcal M(G)(x)+\mathcal M(G)(y)\bigr)^p}{|x-y|^{n+sp-p}}\,dx\,dy.
		\]
		Since $\mathbb B$ is bounded and $p>sp$, with a change of variables, Fubini's theorem, symmetry and \Cref{HLMI}
		\[
		I_{cc}\lesssim \int_{\mathbb B}(\mathcal M(G)(x))^p\,dx
		\lesssim \int_{\mathbb B}G(x)^p\,dx,
		\]
		by the $L^p$-boundedness of $\mathcal M$. Finally,
		\[
		\int_{\mathbb B}G(x)^p\,dx
		=
		\iint_{\mathbb B\times\mathbb B}
		\frac{|E_1f(x)-E_1f(y)|^p}{|x-y|^{n+sp}}\,dx\,dy
		=
		[E_1f]_{W^{s,p}(\mathbb B)}^p.
		\]
		Combining the estimates for $I_{\Om\Om}$, $I_{\Om c}$, and $I_{cc}$ gives
		\[
		[E_2f]_{W^{s,p}(\mathbb B)}\leqa \,[E_1f]_{W^{s,p}(\mathbb B)},
		\]
		as claimed.
	\end{proof}

	The following result is motivated by classical Banach-Zarski\u i theorem (see \cite{Mattila}). We give a self-contained treatment here.
	\begin{lemma}\label{holder-AC}
		Let $E\subset \RR$ and let $f\in C^{0,\alpha}(\RR)\cap C^{1}(\RR\setminus E)$
		for some $\alpha\in (0,1)$. Assume that $\HH^\alpha(E)=0$ and
		\[
		\int_{\RR\setminus E} |f'(x)|\,\d x  < \infty.
		\]
		Then $f$ is absolutely continuous on $\RR$.
	\end{lemma}
	
	\begin{proof}
		Let $K>0$ be a Hölder constant of $f$, i.e.
		\[
		|f(x)-f(y)| \le K |x-y|^\alpha \quad \text{for all } x,y\in\RR.
		\]
		
		Let $\varepsilon>0$. Since $\HH^\alpha(E)=0$, there exists a countable
		collection of open intervals $\{U_i\}_i$ covering $E$ such that
		\[
		\sum_i |U_i|^\alpha < \frac{\varepsilon}{2K}.
		\]
		
		Let $\delta_1 := \frac{\varepsilon}{2K}$.
		
		Also, since $|f'|\in L^1(\RR\setminus E)$, there exists $\delta_2>0$
		such that for every measurable set $A\subset \RR\setminus E$ with
		$|A|<\delta_2$,
		\[
		\int_A |f'(x)|\,\d x  < \frac{\varepsilon}{2}.
		\]
		
		Set $\delta := \min\{\delta_1,\delta_2\}$.
		
		Now let $\{(a_j,b_j)\}_{j=1}^N$ be a finite collection of pairwise disjoint
		intervals with
		\[
		\sum_{j=1}^N (b_j-a_j) < \delta.
		\]
		
		We estimate
		\[
		\sum_{j=1}^N |f(b_j)-f(a_j)|.
		\]
		
		Split each interval into two parts:
		\[
		(a_j,b_j) = \big((a_j,b_j)\cap E\big)\ \cup\ \big((a_j,b_j)\setminus E\big).
		\]
		
		\medskip
		
		\noindent
		\textbf{Contribution from $E$.}
		For each $j$,
		\[
		|f(b_j)-f(a_j)| \le K (b_j-a_j)^\alpha.
		\]
		Hence,
		\[
		\sum_{j=1}^N |f(b_j)-f(a_j)|
		\le K \sum_{j=1}^N (b_j-a_j)^\alpha.
		\]
		Since the intervals intersecting $E$ are contained in $\bigcup_i U_i$, we get
		\[
		\sum_{j:\,(a_j,b_j)\cap E\neq\emptyset} (b_j-a_j)^\alpha
		\le \sum_i |U_i|^\alpha < \frac{\varepsilon}{2K}.
		\]
		Thus,
		\[
		\sum_{j:\,(a_j,b_j)\cap E\neq\emptyset} |f(b_j)-f(a_j)|
		< \frac{\varepsilon}{2}.
		\]
		
		\medskip
		
		\noindent
		\textbf{Contribution from $\RR\setminus E$.}
		On each interval $(a_j,b_j)\subset \RR\setminus E$, the function is $C^1$,
		so by the fundamental theorem of calculus,
		\[
		|f(b_j)-f(a_j)| \le \int_{a_j}^{b_j} |f'(x)|\,\d x .
		\]
		Summing over those intervals,
		\[
		\sum_{j:\,(a_j,b_j)\subset \RR\setminus E} |f(b_j)-f(a_j)|
		\le \int_{\bigcup_j (a_j,b_j)\setminus E} |f'(x)|\,\d x .
		\]
		Since the total length is $<\delta\le \delta_2$, this is $<\varepsilon/2$.
		
		\medskip
		
		Combining both parts,
		\[
		\sum_{j=1}^N |f(b_j)-f(a_j)| < \varepsilon.
		\]
		
		This proves that $f$ is absolutely continuous.
	\end{proof}

	\begin{lemma}\label{slicing}
		Let $E \subset \RR^n$, and let $H \subset \RR^n$ be a fixed $(n-1)$-dimensional hyperplane. For $w \in H$, let $L_w$ denote the line orthogonal to $H$ passing through $w$.
		
		Suppose there exists a set $G \subset H$ with $\HH^{n-1}(G) > 0$ such that
		\[
		\HH^t(E\cap L_w)>0 \quad \text{for all } w \in G.
		\]
		Then
		\[
		\HH^{n-1+t}(E)>0.
		\]
	\end{lemma}
	
	\begin{proof}
		By \Cref{frostman}, for each $w \in G$ there exists a Borel measure $\mu_w$ supported on $E\cap L_w$ such that
		\[
		\mu_w(B_r(x)) \le r^t \quad \text{for all } x \in \RR^n,\ r>0.
		\]

		Define a measure $\mu$ on $\RR^n$ by
		\[
		\mu(A) := \int_G \mu_w(A \cap L_w)\, \d\HH^{n-1}(w).
		\]
		Then $\mu$ is supported on $E$.	For any ball $B_r(x)$ in $\RR^n$, we have
		\begin{align*}
			\mu(B_r(x))
			&= \int_G \mu_w(B_r(x))\, \d\HH^{n-1}(w)\\
			&\leq r^t \int_{\{w\in G\ |\ L_w \cap B_r(x) \neq \emptyset\}} \, \d\HH^{n-1}(w)\\
			&= r^t \HH^{n-1}\lp \{w\in G\ |\ L_w \cap B_r(x) \neq \emptyset\}\rp \\
			&\leq  r^t \HH^{n-1}\lp \pi_H\lp B_r(x)\rp  \rp ,
		\end{align*}
		where $\pi_H$ denotes the orthogonal projection to $H$.
		Thus, we have
		$$
		\mu(B_r(x)) \leq r^{n-1+t}. 
		$$
		Therefore, by \Cref{frostman}, $$\HH^{n-1+t}(E)>0.$$
	\end{proof}

	\begin{proof}[\textbf{Proof of \Cref{th-self-improve}}]
		Let $E_1: \Wdot{1,p}(\Om)\to \Wdot{s,p}(B(0,1))$ be a bounded extension operator. Define the operator $E_2$ as in \cref{eqC.13}. Let $f\in C^\infty(\Om)\cap \Wdot{1,p}(\Om)$, so that, by \Cref{E2-in-fracsob}, $E_2f\in \Wdot{s,p}(\BB)$. We choose a coordinate system which we get by \Cref{loss-sloan-cor}, so that $E_2f\in W^{s,p}(L_w)$ for $\HH^{n-1}$-a.e. $w\in H$ for any coordinate plane $H$. Since $sp>1$, from the hypothesis, and the fact that $L_w$ is one dimensional, we have $E_2f|_{L_w}\in C^{0,\frac{sp-1}{p}}(L_w)$. 
		
		Also, the proof of \Cref{th1} shows that
		\begin{equation*}
			\|\nabla E_2 f\|_{L^p(B(0,1)\setminus \bar{\Om})}\leqa  \|\nabla f(x)\|_{L^{p}(\Om)},
		\end{equation*}
		which implies that 
		\[
		\I{L_w\setminus \partial \Om} \abs{(E_2f)'(x)} \d\HH^1(x) <\infty
		\]
		for $\HH^{n-1}$ -a.e. $w\in H$.
		
		The proof will be complete, in view of \Cref{holder-AC,ACL}, if we can show that, for $\HH^{n-1}$-a.e. $w\in H$, we have $\HH^\frac{sp-1}{p}(L_w \cap \partial \Om)=0$.
		
		If this is not the case, there is some $G\subseteq H$ with $\HH^{n-1}(G)>0$ such that $\HH^\frac{sp-1}{p}(L_w \cap \partial \Om)>0$ for all $w\in G$. \Cref{slicing} then implies that $\HH^{n-1+\frac{sp-1}{p}}(\partial \Om)>0$ for all $w\in G$, which contradicts our hypothesis.
	\end{proof}

	\vspace{1cm}
	\providecommand{\bysame}{\leavevmode\hbox to3em{\hrulefill}\thinspace}
	\providecommand{\noopsort}[1]{}
	\providecommand{\mr}[1]{\href{http://www.ams.org/mathscinet-getitem?mr=#1}{MR~#1}}
	\providecommand{\zbl}[1]{\href{http://www.zentralblatt-math.org/zmath/en/search/?q=an:#1}{Zbl~#1}}
	\providecommand{\jfm}[1]{\href{http://www.emis.de/cgi-bin/JFM-item?#1}{JFM~#1}}
	\providecommand{\arxiv}[1]{\href{http://www.arxiv.org/abs/#1}{arXiv~#1}}
	\providecommand{\doi}[1]{\url{https://doi.org/#1}}
	\providecommand{\MR}{\relax\ifhmode\unskip\space\fi MR }
	\providecommand{\MRhref}[2]{%
		\href{http://www.ams.org/mathscinet-getitem?mr=#1}{#2}
	}
	\providecommand{\href}[2]{#2}

\end{document}